\documentclass[letter,10pt,conference]{ieeeconf} 
\IEEEoverridecommandlockouts                              
\usepackage[utf8]{inputenc}
\usepackage[tbtags]{amsmath}
\usepackage{subfigure}
\usepackage{bbm}
\usepackage[subnum]{cases}
\usepackage{amssymb, psfrag, epstopdf}

\usepackage{amsthm,balance}
\usepackage{enumerate}
\usepackage{graphicx}
\usepackage{color,cite,mathtools,booktabs}
\usepackage{cleveref}
\usepackage{algorithm,algpseudocode}

\newtheorem{lemma}{Lemma}

\newtheorem{assumption}{Assumption}

\DeclareMathOperator{\diag}{dg}
\DeclareMathOperator{\rank}{rank}
\DeclareMathOperator{\trace}{Tr}

\newcommand \mcC{\mathcal{C}}
\newcommand \mcE{\mathcal{E}}

\newcommand \mcG{\mathcal{G}}
\newcommand \hmcG{\hat{\mathcal{G}}}
\newcommand \uX{\underline{X}}
\newcommand \oX{\overline{X}}
\newcommand \mcV{\mathcal{V}}
\newcommand \hmcE{\hat{\mathcal{E}}}

\newcommand \mcZ{\mathcal{Z}}
\newcommand \tL{\tilde{L}}
\newcommand \tW{\tilde{W}}

\title{Designing Power Grid Topologies for Minimizing Network Disturbances:\\
	An Exact MILP Formulation}
\author{Siddharth Bhela, Deepjyoti Deka, Harsha Nagarajan, Vassilis Kekatos} 
	
\begin{document}
	
	\maketitle
	\thispagestyle{empty}
	\pagestyle{empty}

	\begin{abstract}
		The dynamic response of power grids to small transient events or persistent stochastic disturbances influences their stable operation. This paper studies the effect of topology on the linear time-invariant dynamics of power networks. For a variety of stability metrics, a unified framework based on the $H_2$-norm of the system is presented. The proposed framework assesses the robustness of power grids to small disturbances and is used to study the optimal placement of new lines on existing networks as well as the design of radial (tree) and meshed (loopy) topologies for new networks. Although the design task can be posed as a mixed-integer semidefinite program (MI-SDP), its performance does not scale well with network size. Using McCormick relaxation, the topology design problem can be reformulated as a mixed-integer linear program (MILP). To improve the computation time, graphical properties are exploited to provide tighter bounds on the continuous optimization variables. Numerical tests on the IEEE 39-bus feeder demonstrate the efficacy of the optimal topology in minimizing disturbances.
	\end{abstract}
	
	\section{Introduction}
	\label{sec:introduction}
	\allowdisplaybreaks
	The electric power system is continually changing. It is expected that the grid of the future will have higher variability due to renewables, changing load patterns, and distributed energy sources~\cite{gayme}. This paradigm shift will pose an enormous challenge for design and stable operation of power networks. The inherent uncertainty associated with renewable energy sources and active loads is likely to produce more frequent and higher amplitude disturbances \cite{gayme}. In addition, owing to the lower aggregate inertia of systems with high penetration of renewables, the capability of power networks to handle such disturbances may be significantly reduced \cite{ulbig2014impact}.
	
	Thus, improving the dynamic performance of the power grid is of importance and has received greater attention from academia and industry. Efforts in this direction include development of payment structures and novel markets, as well as analysis of techniques to incentivize load-side participation \cite{seanmeyn}, \cite{Trudnowski06}, \cite{ferc755}, \cite{ercotwhitepaper}. The benefits of load-side controllers has motivated a series of recent works to understand how different system parameters and controller designs impact the transient response of the network \cite{mallada16}, \cite{GuoLow18}. Recent work in \cite{Poolla17} has also explored the optimal placement of virtual inertia to improve stability.
	
	Compared to the previously mentioned system parameters, the effect of network topology on transient stability is less well understood. Without detailed simulations, it is usually hard to infer how a change in network topology influences the overall grid behavior and performance. Recent work in \cite{GuoLow18} shows that the impact of network topology on the power system can be quantified through the network Laplacian matrix eigenvalues. In addition, grid robustness against low frequency disturbances is mostly determined by network connectivity \cite{GuoLow18}, further motivating this study. Past studies in the power and control systems communities have also looked at designing network topologies for specific goals using system theoretic tools. Such goals include reduction of transient line losses \cite{gayme}, improvement in feedback control \cite{john}, \cite{mallada}, coherence based network design \cite{SDPvoltage} and augmentation \cite{lygeros}. Semidefinite programming (SDP) based tools have also been utilized to design and augment network topologies for dynamic control~\cite{arpita}, \cite{SDPvoltage}, \cite{nagarajan2017optimal}.
	
	While the primary focus here is topology design, we recognize that there is line of related work dealing with learning network topologies and line parameters \cite{GVS18}, \cite{Guido2018}, \cite{dekatcns}. Schemes that rely on passive data have been used in \cite{dekatcns} and \cite{GVS18} for learning radial topologies. Different from them, the work in \cite{Guido2018} actively probes the grid to recover radial topologies and verify line statuses.

	In this paper we are interested in studying the effect of topology on the power grid dynamics. For a variety of objective functions, such as line loss reduction, fast damping of oscillations, and network coherence, our previous work \cite{Deka18ACC} presented a unified framework to study topology design based on the $\mathcal{H}_2$-norm. In \cite{Deka18ACC}, the focus was on topology reconfiguration rather then topology design. Further, the work in \cite{Deka18ACC} developed suboptimal algorithms, albeit with guarantees on optimality gap, to tackle the combinatorial design problems involved. Here, we present reformulations of the topology design task that allow us to solve the problem to optimality.
	
	
	Our contributions are as follows. First, we provide a comprehensive modeling and analysis framework for the topology design problem to optimize a $\mathcal{H}_{2}$ norm based performance metric subject to budget constraints in Section~\ref{sec:formulation}. Second, in Section~\ref{sec:top} we show that although the topology design task is inherently non-convex, it is possible to exactly reformulate the problem in tractable form using McCormick relaxation (or linearization). This can then be used with off-the-shelf solvers to determine the optimal solution. Further, we show that exploiting graph-theoretic properties to tighten bounds on the continuous optimization variables yields significant improvements in computation time. Section~\ref{sec:nx} discusses numerical tests based on the IEEE $39-$bus test case followed by conclusions and future directions in Section~\ref{sec:conc}.

	\textbf{Notation:} Column vectors (matrices) are denoted by lower- (upper-) case letters and sets by calligraphic symbols, unless noted otherwise. The cardinality of set $\mathcal{X}$ is denoted by $|\mathcal{X}|$. Given a real-valued sequence $\{x_1, x_2, \ldots, x_N\}$, $x \in \mathbb{R}^N$ is the vector obtained by stacking the scalars $x_i$ and $\diag(\{x_i\})$ is the corresponding diagonal matrix. The operator $(\cdot)^{\top}$ stands for transposition. The $N$-dimensional all ones vector is denoted by $1_N$; $I_{N}$ is the $N \times N$ identity matrix; and $e_i$ is the canonical vector with a $1$ at the $i$-th entry and zero everywhere else. The notation $\dot{\theta_i}$ denotes its time derivative $\frac{\delta \theta_i}{\delta t}$.
	
	\section{Grid Modeling and Network Coherence}
	\label{sec:formulation}
	This section presents a linearized power system model, defines network coherence metrics, and reviews their connection to studying the stability of a linear dynamical system.

	\subsection{Dynamic Power System Model}\label{subsec:dynamic}
	A power network having $N+1$ buses can be modeled as a connected graph ${\cal G }= ({\cal V}, \cal{E})$,  whose nodes ${\cal V}:= \{0,1,\ldots,N\}$ correspond to buses, and edges ${\mcE} \subseteq \cal V \times \cal V$ to undirected lines. Bus $i=0$ is selected as the reference; all other buses constitute set $\mathcal{V}_r:=\mathcal{V}\backslash\{0\}$. Let $b_{ij}>0$ be the susceptance of line $(i,j)\in \mcE$ connecting nodes $i$ and $j\in \cal V$. Then, the susceptance Laplacian matrix $L \in \mathbb{R}^{(N+1) \times (N+1)}$ associated with the grid graph $\cal G$ can be defined as
	\begin{align*}
		L_{ij}: = \begin{cases}-b_{ij} &, ~\text{if $(i,j) \in {\mcE}$}\\
			\sum_{(i,j)\in {\mcE}}b_{ij} &,~\text{if $j=i$}\\
			0 &,~\text{otherwise}. \end{cases} 
	\end{align*}
	
	Each node $i \in \mathcal{V}$ is associated with a phase angle $\theta_i$, frequency $\omega_i= \dot{\theta_i}$, inertia constant $M_i$, and damping coefficient~$D_i$; see~\cite{kundur}. Since bus $i$ may host an ensemble of devices such as synchronous machines, renewable or energy storage sources, frequency-dependent or actively controlled frequency-responsive loads, the parameters $M_i$ and $D_i$ characterize their aggregate behavior \cite{Poolla17}. 
	
	Focusing on small-signal stability, the quantities $(\theta_i,\omega_i)$ will henceforth refer to the deviations of nodal voltage angles and frequencies from their nominal values. The grid dynamics at bus $i$ can then be described by the linearized swing equation \cite{kundur}
	\begin{equation}
		M_i\dot{\omega}_i + D_i\omega_i= P^m_i-P_{i}^e  +  u_i \label{swing}
	\end{equation}    
	where $P^m_i$ denotes the mechanical power input; $P_i^e$ is the electric power flowing from bus $i$ to the grid; and $u_i$ is the power consumed at bus $i$. Again, the aforementioned quantities refer to the deviations from their scheduled values. Under the linearized DC model, the power flowing from bus $i$ to the grid can be approximated as \cite{kundur}
	\begin{equation}
		P_{i}^e  \simeq \sum_{(i,j)\in {\mcE}}b_{ij}(\theta_i-\theta_j). \label{DCpowerflow}
	\end{equation} 
	
	Combining \eqref{swing} and \eqref{DCpowerflow}, the state-space representation of the power grid is
	\begin{align} \label{swingmatrixfull}
		\begin{bmatrix} \dot{{\theta}} \\ \dot{{\omega}}  \end{bmatrix}   = \underbrace{\begin{bmatrix} 0 & {I}_N\\ -{M}^{-1}L  & -{M}^{-1}{D} \end{bmatrix}}_{A:=} \begin{bmatrix} {\theta} \\ {\omega}  \end{bmatrix} + \underbrace{\begin{bmatrix} 0\\ {M}^{-1}\end{bmatrix}}_{B:=} {u}
	\end{align}
	where ${M}:=\diag(\{M_i\})$ and ${D}:=\diag(\{D_i\})$ are diagonal matrices containing the inertia and damping coefficients; the states ${\omega} \in \mathbb{R}^{N+1}$ and ${\theta} \in \mathbb{R}^{N+1}$ are accordingly the stacked vectors of nodal frequencies and angles; and ${u} \in \mathbb{R}^{N+1}$ is the vector of local power disturbances. The subsequent analysis relies on the ensuing assumption.
	
	\begin{assumption}\label{as:A1}
		The inertia coefficients are strictly positive and damping coefficients are identical for all buses, that is $M_i > 0$ and $D_i = d$ for all $i\in \cal V$.
	\end{assumption}
	
	The non-zero inertia assumption is not necessary, but simplifies our presentation. If $M_i=0$ for a bus $i\in \cal V$, then a Kron-reduced network containing only nodes with inertia can be obtained. The second assumption of constant damping has been previously used in~\cite{gayme}, \cite{john}. When it is not satisfied, the stability metric defined in the next section does not enjoy a closed-form expression, and only bounds can be derived. In our future work, we plan to extend our approach to the case with variable $D$. 
	
	\subsection{Generalized Network Coherence Metrics}
	Given the state-space model in \eqref{swingmatrixfull}, our goal is to design network topologies or augment existing ones to minimize the voltage angle deviations caused by load disturbances. These angle deviations are formally captured by the metric of \emph{network coherence}~\cite{bamieh12}. The latter is defined as $\lim_{t\rightarrow \infty}\mathbb{E}[f_c(t)]$ where $f_c(t)$ is the steady-state deviation of the angles from their grid-average
	\begin{equation} \label{coherence}
		f_c(t):=\sum_{i=1}^{N}\left(\theta_i(t) - \frac{1}{N}\sum_{j=1}^{N} \theta_j(t) \right)^2.
	\end{equation}
	In other words, network coherence quantifies how tightly bus angles drift together. Larger variances in angle deviations reflect a more disordered network~\cite{gayme}. 
	
	
	Instead of network coherence, the system operator may be interested in minimizing the expected steady-state value of a generalized function combining both voltage angle and frequency excursions~\cite{Deka18ACC}, \cite{Poolla17}
	\begin{equation}\label{loss}
		f(t) := \sum_{\forall i \neq j} w_{ij}(\theta_i(t)-\theta_j(t))^2 + \sum_{i \in \mathcal{V}}s_{i}\omega_i^2(t)
	\end{equation}
	where $w_{ij}$ and $s_i$ are given non-negative scalars. The weights $w_{ij}$ induce a connected weighted graph $\mcG_w$ that is not necessarily identical to $\cal G$; see Fig.~\ref{fig:concept}. Let $W$ be the Laplacian matrix of graph $\mcG_w$ and $S:=\diag(\{s_i\})$. Then, it is not hard to verify that $f(t) = \|y(t)\|_2^2$ where
	\begin{align} \label{output}
		y(t) :=  \underbrace{\begin{bmatrix}W^{1/2}  ~&0 \\0 ~&S^{1/2}\end{bmatrix}}_{C:=}  \begin{bmatrix}\theta(t)\\ \omega(t)\end{bmatrix}.
	\end{align}
	
	\begin{figure}[t]
		\centering
		\includegraphics[scale=0.8]{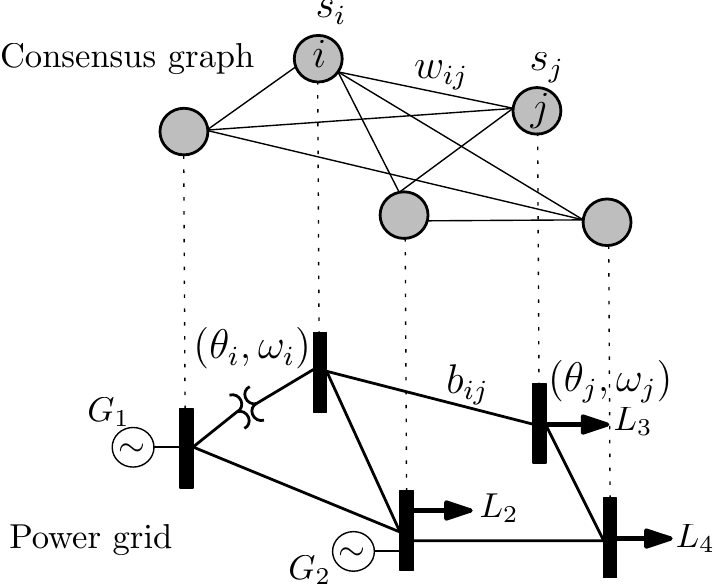}
		\caption{Illustration of power grid and its associated coherence graph.\cite{Deka18ACC}}
		\label{fig:concept}
	\end{figure}
	
	Being a Laplacian matrix, matrix $W$ is positive semidefinite, and so its matrix square root $W^{1/2}$ is well-defined. The matrix $S^{1/2}$ is well-defined too. The importance of the generalized coherence metric of \eqref{loss} is that for different choices of~$(W,S)$, one can represent different grid performance metrics and study them in a unified manner~\cite{Deka18ACC}: For instance, if the system operator is only interested in minimizing frequency excursions, one can set $W=0$ and $S=I_{N+1}$. Similarly, if the goal is to reduce transient line losses, one can choose $W=L$ and $S=0$. Lastly, the network coherence metric of \eqref{coherence} corresponds to the case of $W = I_{N+1}- \frac{1}{N+1}1_{N+1}1_{N+1}^{\top}$ and $S =0$. Figure~\ref{fig:concept} shows the physical power network and its associated coherence graph. Note that the choice of $(W,S)$ for network coherence penalizes global deviations, whereas that for line loss reduction penalizes local deviations. 
	
	
	\subsection{Relation to Stability Analysis}
	The expected steady-state value of $f(t)$ can be interpreted as the squared ${\cal H}_2$-norm of the linear time-invariant (LTI) system described by \eqref{swingmatrixfull} and \eqref{output}. This system will be compactly denoted by $H:=(A,B,C)$. Leveraging this link, the generalized network coherence is amenable to a closed-form expression~\cite{Deka18ACC}. 
	
	The ${\cal H}_2$ norm is widely used as a stability performance metric and has several interpretations~\cite{boyd1991linear}: For unit-variance stochastic white noise $u(t)$, the ${\cal H}_2$-norm of an LTI system equals the steady-state output variance~\cite[Ch.~5]{boyd1991linear}, \cite{gayme}, \cite{Poolla17}
	\begin{equation} \label{h3}
		\|H\|^2_{{\cal H}_2}:=\lim_{t\rightarrow \infty}\mathbb{E}\left[\|y(t)\|_2^2\right].
	\end{equation}
	For unit-impulse disturbances $u_i(t) =e_i\delta(t)$ for $i=1,\ldots,N$, the ${\cal H}_2$-norm can be equivalently written as
	\begin{equation} \label{h2}
		\|H\|^2_{{\cal H}_2}:=\sum_{i=1}^{N}\int_{0}^{\infty} \|y_i(t)\|_2^2 dt
	\end{equation}
	where $y_i(t)$ is the system output corresponding to disturbance vector $u_i(t)$. Instead of evaluating the expectation in~\eqref{h3} or the time integral in \eqref{h2}, the $\mathcal{H}_2$-norm for system $H$ can be expressed as~\cite[Ch.~5]{boyd1991linear} 
	\begin{equation} \label{obsh2}
		\|H\|^2_{{\cal H}_2} = \trace(B^{\top}QB)
	\end{equation}
	where $Q:=\int_{0}^{\infty}e^{A^{\top}t}C^{\top}Ce^{At} dt$ is the observability Gramian matrix of the LTI system $H$. In fact, the matrix $Q$ can be computed as the solution to the linear Lyapunov equation~\cite[Ch.~5]{boyd1991linear} 
	\begin{equation}\label{observability}
		A^{\top}Q+QA = -C^{\top}C. 
	\end{equation}
	Matrix $Q$ is known to be symmetric positive semidefinite, so it can be partitioned as
	\begin{equation} \label{Q}
		Q=
		\begin{bmatrix}
			Q_1 & Q_0 \\
			Q_0^{\top} & Q_2
		\end{bmatrix}.
	\end{equation}
	Based on the reformulation in~\eqref{obsh2}, we next design topologies attaining minimum generalized network coherence.
	
	\section{Optimal Grid Topology Design} \label{sec:top}
	Among other criteria, the topology of a grid can be designed to minimize the generalized network coherence metric of \eqref{loss}. Given a graph $\hat{\mathcal{G}}=(\mcV, \hat{\mcE})$, where $\hat{\mathcal{E}}$ is the set of candidate lines weighted by their susceptances, the goal is to find the subset ${\mcE} \subseteq \hmcE$ of cardinality $|\mathcal{E}|\leq K$ with $K\geq N$ attaining the minimum generalized network coherence. Given the equivalences of the previous section, this task can be posed as the optimization problem
	\begin{subequations}\label{eq:P1}
		\begin{align}
			\arg \min_{\mathcal{E} \in \hat{\mathcal{E}}}~&~\trace(B^{\top}QB)\\
			\mathrm{s.to}~ &~|{\mcE}| \leq K\label{eq:P1:cardinality}\\
			&~Q \text{~satisfies}~\eqref{observability}\\
			&~{\mcE} \text{~is connected}.
		\end{align}
	\end{subequations}
	The constraint in \eqref{eq:P1:cardinality} reflects the budget on the number of edges. For $K=N$, the problem in \eqref{eq:P1} yields a tree topology, which is important for typically radially operated distribution grids. Interestingly, leveraging the problem structure under Assumption~\ref{as:A1}, it can be shown that the objective of \eqref{eq:P1} simplifies as~\cite{Poolla17}, \cite{Deka18ACC}, \cite{gayme} 
	\begin{equation}\label{eq:cost2}
		\trace(B^{\top}QB)=\frac{\trace(WL^+) + \trace(SM^{-1})}{2d}.
	\end{equation}
	If Assumption~\ref{as:A1} is not met, i.e., the damping coefficients are not identical $(D\neq dI_N)$, then it may not be possible to find a closed-form expression for the objective in \eqref{eq:P1}. If the damping coefficients are known to lie within $[d_{\min},d_{\max}]$, then the objective of \eqref{eq:P1} can be bounded as $\frac{\trace(WL^+) + \trace(SM^{-1})}{2d_{\max}}\leq \trace(B^{\top}QB)\leq \frac{\trace(WL^+) + \trace(SM^{-1})}{2d_{\min}}$; see~\cite{gayme}, \cite{Poolla17}. Therefore, as the range $[d_{\min},d_{\max}]$ becomes narrower, minimizing the numerator of \eqref{eq:cost2} approaches the optimal solution to \eqref{eq:P1}. From Assumption~\ref{as:A1}, we consider $d_{\min}=d_{\max}$ here.
	
	The second summand in the right-hand side of \eqref{eq:cost2} is independent of the grid topology, and can thus be ignored. Problem~\eqref{eq:P1} can then be reformulated as
	\begin{align}\label{eq:P2}
		\arg\min_{{\mcE} \in \hmcE} ~&~\trace(WL^+)\\
		\mathrm{s.to}~ &~ |{\mcE}| \leq K \nonumber\\
		~&~\rank(L) = N \nonumber
	\end{align}
	where the rank constraint ensures that the graph induced by $\mathcal{E}$ is connected. Problem~\eqref{eq:P2} could be tackled with brute-force algorithms over all the possible topologies of budget $K$ and below; though that would incur exponential complexity. 
	
	The objective in \eqref{eq:P2} can be written in terms of the inverse of the \emph{reduced Laplacian} matrix of $\mcG$ as explained next. The claim has appeared in \cite[Lemma~3.2]{gayme}, albeit for the restricted case where $W$ and $L$ have the same structure.
	\begin{lemma}\label{le:cost}
		If $\tW$ and $\tL$ are the $N\times N$ matrices obtained after removing the first row and column from $W$ and $L$, respectively, then
		\[\trace(WL^+)=\trace(\tW\tL^{-1}).\]
	\end{lemma}
	
	\begin{proof}
		The Laplacian matrices can be described in terms of their reduced counterparts as
		\begin{equation}\label{eq:reduced}
			W=R\tW R^\top \quad \text{and}\quad L=R\tL R^\top
		\end{equation}
		where $R:=[1_N~-I_N]^\top$. This is because a Laplacian matrix satisfies $L1_{N+1}=0_{N+1}$. If we define matrix $J:=I_N-\frac{1}{N+1}1_N1_N^\top$, then it is not hard to see that
		\begin{equation}\label{eq:RR}
			R^\top R=I_N+1_N1_N^\top=J^{-1}.
		\end{equation}
		Then the pseudoinverse of $L$ can be expressed as
		\begin{equation}\label{eq:pseudo}
			L^+=RJ\tL^{-1} J R^\top.
		\end{equation}
		The latter can be shown by simply verifying that $LL^+L=L$ and $L^+LL^+=L^+$. From \eqref{eq:reduced} and \eqref{eq:pseudo}, we get that
		\[\trace(WL^+)=\trace\left( R\tW R^\top R J \tL^{-1} J R^\top \right)=\trace(\tW\tL^{-1})\]
		where the last equality follows from \eqref{eq:RR} and the cyclic property of the trace.
	\end{proof}
	
	To express the optimization in \eqref{eq:P2} over $\hmcE$ in a more convenient form, let us introduce a binary variable $z_m$ for every line $m\in\hmcE$. This variable is $z_m=1$ if line $m$ is selected, i.e., $m\in\mcE$; and $z_m=0$, otherwise. If we stack variables $\{z_m\}_{m\in\hmcE}$ in vector $z$, then $z$ has to lie in the set
	\begin{equation}\label{eq:Z}
		\mcZ:=\left\{z:z^{\top}{1}_{|\hat{\mathcal{E}}|} \leq K, \ z \in \{0,1\}^{|\hmcE|}\right\}.
	\end{equation}
	
	Based on the line selection vector $z$, the reduced susceptance Laplacian of $\mcG$ can be expressed as
	\begin{equation}\label{eq:L2}
		\tL(z)=\sum_{(i,j)\in\hmcE} z_{ij}b_{ij} a_{ij} a_{ij}^\top
	\end{equation}
	where each vector $a_{ij}$ corresponds to line $(i,j)\in\hmcE$, and its $n$-th entry is defined as
	\begin{equation*}
		[a_{ij}]_n: = \begin{cases}
			+1 &,\text{ if }n=i\\
			-1 &,\text{ if }n=j\\
			0 &,\text{ otherwise.}\\
		\end{cases} 
	\end{equation*}
	
	Given Lemma~\ref{le:cost} and \eqref{eq:L2}, the optimization in \eqref{eq:P2} can be equivalently written as the mixed-integer semidefinite program (MI-SDP)
	\begin{subequations} \label{eq:P3}
		\begin{align}
			\arg\min_{X,z\in \mcZ} ~&~\trace(\tW X)\label{eq:P3:cost}\\
			\mathrm{s.to}~&~\begin{bmatrix}
				X & I_N \\
				I_N & \tL(z)
			\end{bmatrix}\succeq 0.\label{eq:P3:con}
		\end{align}
	\end{subequations}
	To see this, the constraint in \eqref{eq:P3:con} is equivalent to $X\succeq 0$ and $X\succeq \tL^{-1}(z)$; see~\cite[Sec.~A.5.5]{BoVa04}. Since $\tW\succ 0$, the optimal $X$ can be shown to be $X=\tL^{-1}(z)$. In fact, constraint \eqref{eq:P3:con} waives the possibility of the optimal $\tL$ being singular, and thus, ensures the connectedness of the graph. 
	
	To overcome the computational complexity of the MI-SDP in \eqref{eq:P3}, one may relax the binary variables to box constraints as $z\in [0,1]^{|\hmcE|}$ to get an ordinary SDP, which can be handled by off-the-shelf solvers for moderately-sized networks. Being a relaxation, the SDP solution provides a lower bound on the cost of \eqref{eq:P3}. If the obtained solution of the SDP turns out to be binary, then this $z$ minimizes the MI-SDP in \eqref{eq:P3} as well. Otherwise, (randomized) rounding heuristics can be adopted to acquire a feasible $z$. 
	
	Aiming at an exact solver, we will next show how the MI-SDP of \eqref{eq:P3} can be equivalently formulated as an MILP. To this end, we first rewrite \eqref{eq:P3} as
	\begin{subequations} \label{eq:P4}
		\begin{align}
			(X^*,z^*)\in\arg\min_{X,z\in \mcZ} ~&~\trace(\tW X)\label{eq:P4:cost}\\
			\mathrm{s.to}~ &~\tL(z)X = I_N.\label{eq:P4:con}
		\end{align}
	\end{subequations}
	Note that for constraint \eqref{eq:P4:con} to hold, $\tL(z^*)$ must be non-singular and $X^*=\tL^{-1}(z^*)$. Although its cost is linear, problem \eqref{eq:P4} is non-convex due to the bilinear constraints in \eqref{eq:P4:con} and because vector $z$ is binary. To handle the former, we adopt the McCormick relaxation technique~\cite{mccormick1976computability}, which is briefly reviewed next.
	
	Constraint~\eqref{eq:P4:con} involves the bilinear terms $z_{m}X_{ij}$ for $m\in\hmcE$ and $i,j\in\mcV_r$. For each such term, introduce a new variable 
	\begin{equation} \label{bilinear}
		y_{mij}=z_{m}X_{ij}.
	\end{equation}
	Suppose that the entries $X^*_{ij}$ are known to lie within the range $[\uX_{ij},\oX_{ij}]$. Since $z_m\in [0,1]$, the ensuing inequalities hold trivially~\cite{mccormick1976computability} 
	\begin{subequations} \label{mcm1}
		\begin{align}
			& z_{m}(X_{ij} - \uX_{ij})  \geq 0\\
			& (z_{m}-1)(X_{ij} - \oX_{ij}) \geq 0 \\
			& z_{m}(X_{ij} - \oX_{ij}) \leq 0 \\
			& (z_{m}-1)(X_{ij} - \uX_{ij}) \leq 0.  
		\end{align}
	\end{subequations}
	Substituting $z_{m}X_{ij}$ by $y_{mij}$ in \eqref{mcm1} provides
	\begin{subequations}\label{mccormick}
		\begin{align}
			& y_{mij} \geq z_{m}\uX_{ij}\label{mccormick:a}\\ 
			& y_{mij} \geq {X}_{ij} + z_{m}\oX_{ij} -\oX_{ij}\label{mccormick:b}\\ 
			& y_{mij} \leq z_{m}\oX_{ij}\label{mccormick:c}\\ 
			& y_{mij} \leq {X}_{ij} + z_{m}\uX_{ij} -\uX_{ij}.\label{mccormick:d}
		\end{align}
	\end{subequations}
	One can replace the bilinear terms in \eqref{eq:P4:con} by $y_{mij}$'s and enforce \eqref{bilinear} and \eqref{mccormick} as additional constraints for all $m\in\hmcE$ and $i,j\in\mcV_r$. In that case, the constraints \eqref{mccormick} are apparently redundant. However, one may simplify the problem by dropping \eqref{bilinear} to get an MILP reformulation of \eqref{eq:P4}. Interestingly, this reformulation is \emph{exact} due to the binary nature of $z$. To see this, if $z_m^*=1$ for some $m$, then \eqref{mccormick:b} and \eqref{mccormick:d} imply $y_{mij}^*=X^*_{ij}$ for all $i,j\in\mcV_r$. Otherwise, if $z_m^*=0$, then \eqref{mccormick:a} and \eqref{mccormick:c} imply $y_{mij}^*=0$ for all $i,j\in\mcV_r$. 
	
	Through the aforementioned process, problem \eqref{eq:P4} has been converted to an MILP over the variables $\{X_{ij}\}$, $\{z_m\}$, and $\{y_{mij}\}$. MILPs can be handled by state-of-the-art solvers such as Gurobi~\cite{gurobi}, and are known to scale better than MI-SDPs. Recall that the MILP reformulation of \eqref{eq:P4} requires the bounds $(\uX_{ij},\oX_{ij})$ on each $(i,j)$-th entry of $X^*$. Lacking prior information on $X^*$, one could select an arbitrarily wide range $[\uX_{ij},\oX_{ij}]$. However, this could slow down the MILP solver significantly. In the other extreme, if $X^*$ is known, that is $\uX_{ij}=\oX_{ij}$ for all $i,j\in\mcV_r$, then the binary vector $z$ can be recovered by simply solving the linear equations in \eqref{eq:P4:con}. To improve the run time of the involved MILP, we next derive tighter, non-trivial bounds on the entries of $X_{ij}^*$.

	\section{Bounding the Continuous Variables}\label{sec:bounds}
	Depending on the problem structure, different bounds can be derived on $X^*_{ij}$s. This section considers two classes of topology design tasks. In the first task, some lines are already energized and the operator would like to augment a connected network by additional lines to further improve its stability. In the second task, a network topology is designed from scratch with the additional requirement of a radial grid.

	\subsection{Augmenting Existing Networks} \label{subsec:augment}
	Given the graph $\hmcG=(\mcV, \hmcE)$, this problem setup considers a pre-existing connected network described by $\mcG_e=(\mathcal{V},\mathcal{E}_e)$, and the goal is to energize additional lines from $\hmcE\setminus \mcE_e$ to improve its stability. In essence, this corresponds to the problem in \eqref{eq:P4} with the entries of $z$ corresponding to the lines in $\mcE_e$ being set to one. Then, based on \eqref{eq:L2}, the reduced Laplacian matrix of the existing network is obviously
	\begin{equation*}
		\tL_e:=\sum_{(i,j)\in\mcE_e} b_{ij} a_{ij} a_{ij}^\top.
	\end{equation*}
	
	Under this setup, the entries of $X^*$ minimizing \eqref{eq:P4} under the additional constraints $z_\ell=1$ for all $\ell\in\mcE_e$, can be bounded as follows.
	
	\begin{lemma}\label{le:augment}
		The entries of $X^*$ are bounded by
		\begin{equation}\label{eq:augment}
			0 < X^*_{ij} \leq \frac{[\tL_e^{-1}]_{ii} + [\tL_e^{-1}]_{jj}}{2},\quad \forall(i,j) \in \mcV_r.
		\end{equation}
	\end{lemma}
	
	\begin{proof}
		Observe that $\tL_e$ can be written as $\tL(z)$ by fixing the entries of $z$ corresponding to the lines in $\mcE_e$ to $1$. Since the same entries will remain $1$ in $z^*$, it readily follows that $\tL(z^*) \succeq \tL_e\succ 0$, where $\tL_e$ is non-singular since the existing network is already connected. Then, it follows that $X^* \preceq \tL_e^{-1}$ and $v^{\top}(X^*-\tL_e^{-1})v < {0}$ for any $v\neq {0}$. Setting $v=e_i$, the diagonal entries of $X^*$ can be bounded as $X^*_{ii} \leq [\tL_e^{-1}]_{ii}$ for all $i\in\mcV_r$. Since $X^*\succ 0$ by constraint \eqref{eq:P4:con}, we have that $({e}_i-{e}_j)^{\top}X^*({e}_i-{e}_j)>0$ for all $i,j\in\mcV_r$. The latter provides $X^*_{ij}<\frac{1}{2}(X^*_{ii}+X^*_{jj})$. Upper bounding the diagonal entries with the bounds obtained earlier proves the upper bound in \eqref{eq:augment}. The lower bound in \eqref{eq:augment} can be trivially obtained since $L(z^*)$ is an M-matrix, and so its inverse has positive entries.
	\end{proof}
	
	The reduced Laplacian matrix $\tL_e$ of the existing network $\mcG_e$ is invertible as long as $\mcG_e$ is connected. If that is not the case, one could obtain bounds on $X^*_{ij}$'s by imposing a radial structure on the sought topology as discussed next.

	\subsection{Radial Topology Design}
	The setup considered here designs a network afresh, but under the requirement that it is radial. The analysis simplifies under the following mild assumption.
	
	\begin{assumption} \label{as:A2}
		There exists a node in $\mcV$ that is incident to exactly one edge.
	\end{assumption}
	
	To derive bounds on the entries of $X^*$ minimizing \eqref{eq:P4} for the special case of $K=N$ (radial network), let us first construct the graph $\mcG_r=(\mcV, \hmcE_r)$, where $\hmcE_r$ consists of the edges in $\hmcE$, but with inverse weights $x_{ij}:=b_{ij}^{-1}$ for all $(i,j)\in\hmcE$. Based on $\mcG_r$, we define some additional properties that will be useful later. If one of the nodes in $\mcV$ satisfying Assumption~\ref{as:A2} is selected as the reference node, then the weight (inverse susceptance) of its incident line is denoted by $x_0$. Let us also define the minimum of the inverse line susceptances
	\[x_{\min}:=\min_{(i,j)\in\hmcE_r} x_{ij}.\]
	
	Before solving \eqref{eq:P4}, we find the maximum spanning tree of graph $\mcG_r$, and denote the sum of its edge weights by $f$. The maximum spanning tree can be found efficiently by finding the minimum spanning tree on $\mcG_r$ upon negating its edge weights \cite{cormen}. Lastly, for each node $i\in\mcV_r$, we find its shortest path to the reference node $0$ in $\mcG_r$. The sum of edge weights along this shortest path will be denoted by $h_i$. 
	
	\begin{lemma}\label{UBcor1}
		Under Assumption~\ref{as:A2}, the entries of $X^*$, minimizing \eqref{eq:P4} for $K=N$, are bounded as
		\begin{subequations}\label{eq:radial}
			\begin{align}
				&h_i \leq X^*_{ii} \leq f,\quad \forall i \in \mcV_r \label{eq:radial:d}\\
				&x_0 \leq X^*_{ij} \leq f-x_{\min},\quad \forall i,j \in \mcV_r.\label{eq:radial:od}
			\end{align}
		\end{subequations}
	\end{lemma}
	
	\begin{proof}
		The bounds rely on a fundamental property of the inverse Laplacian matrix of a radial network: If $\tL$ is the reduced susceptance Laplacian of a radial network and $X:=\tL^{-1}$, then the entry $X_{ij}$ equals the sum of the inverse susceptances that are common to the paths of nodes $i$ and $j$ to the reference bus; see \cite[Lemma~1]{dekasmartgridcomm}. Under Assumption~\ref{as:A2}, the common path of any pair of nodes $(i,j) \in \mcV_r$ must include at least the line incident to the reference bus, and hence, $X^*_{ij}\geq x_0$. By the definition of shortest path, the entry $X^*_{ii}$ is lower bounded by $h_i$ for all $i\in\mcV_r$, thus providing the lower bound in~\eqref{eq:radial:d}.
		
		
		Regarding the upper bound in \eqref{eq:radial:d}, recall that the $(i,i)$-th entry of $X^*$ equals the sum of weights on the path from $i$ to the reference node $0$. The sum of weights on the longest such path is still upper bounded by the sum weight $f$ of all edges in the maximum spanning tree. Because the entry $X^*_{ij}$ for $i \neq j$ must have at least one less edge than the longest path, the upper bound in \eqref{eq:radial:od} holds as well. 
	\end{proof}
	Note that the upper bounds can be tight in the setting where the maximum spanning and optimum trees are the same line graph.

	\subsection{Model Simplification and Bound Tightening} \label{subsec:boundt}
	Exploiting the structure of $\hmcG=(\mcV,\hmcE)$ can provide additional information on the bounds of $X^*$ matrix entries to accelerate solving \eqref{eq:P4} for the special case of $K=N$ (radial grid). For example, if $\hmcG$ gets disconnected upon removing edge $\ell\in\hmcE$, this edge $\ell$ belongs to the sought tree topology and $z_\ell^*=1$ before solving \eqref{eq:P4}. To identify such edges, we resort to the notion of a \emph{graph cutset} $\mcC\subset\hmcE$, defined as a subset of edges that once removed, splits the graph $\hmcG$ into two or more  connected components. The edges in this cutset will be termed as \emph{critical edges}. Now, we present a simple algorithm to enumerate all the single critical edges $(|\mathcal{C}| = 1)$ by initializing the graph $\hmcG$'s weights to unit values. 
	
	\begin{enumerate}[Step 1:]
		\item Solve the min-cut problem on $\hmcG$ with unit edge weights by using the standard max-flow min-cut algorithm ~\cite{Ford}.
		\item Increase the weight for the edge labelled as critical to $1+\epsilon$ for some $\epsilon>0$.
		\item If $|\mathcal{C}|>1$, quit; else, go to Step 1.
	\end{enumerate}
	The second step ensures that every time we are identifying a new critical edge. Upon completing this process, the entries of $z^*$ corresponding to the critical edges can be safely set to $1$. This process not only reduces the binary search for $z^*$ in \eqref{eq:P4}, but it further tightens the lower bounds on certain $X^*_{ij}$'s as discussed in Lemma \ref{UBcor2}; see also Fig.~\ref{fig:mincut}. 
	
	\begin{figure}[t]
		\centering
		\subfigure[]{
			\includegraphics[scale=0.66]{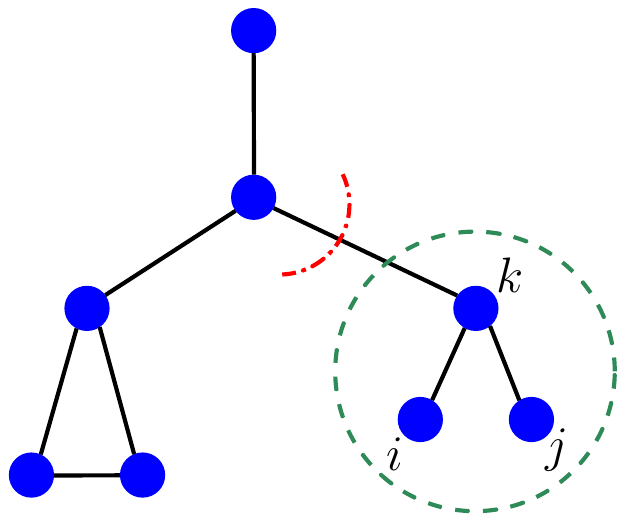}}       
		\subfigure[]{
			\includegraphics[scale=0.66]{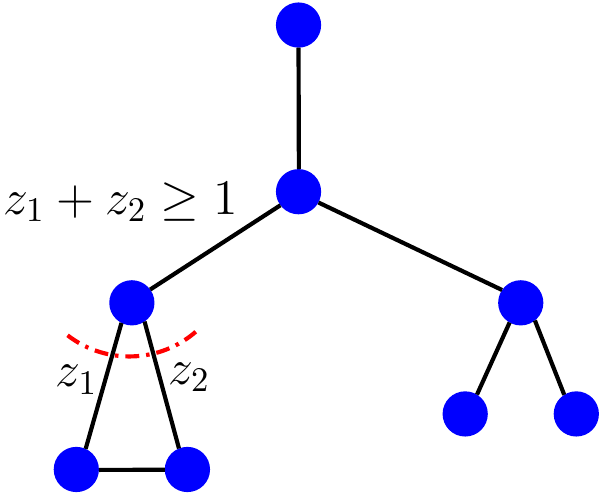}} 
		\caption{Grid graph ${\hmcG}=(\mcV,\hmcE)$ with candidate location of edges. \emph{Left panel:} The lower bounds on $X^*_{ki}$, $X^*_{kj}$, and $X^*_{ij}$ can be tightened using the shortest path weight $h_k$ due to the critical edge colored in red. \emph{Right panel:} The size--$2$ cutset shown in red implies the constraint $z_1+z_2 \geq 1$.}
		\label{fig:mincut}
	\end{figure}
	
	\begin{lemma}\label{UBcor2}
		Suppose a critical edge $\ell=(i,j) \in \hmcE$ partitions the nodes of $\hmcG$ into two disjoint connected components $\mcV_\ell$ and its complement $\bar{\mcV}_\ell$. If $\mcV_\ell$ contains node $i$ as well as the reference bus, then \eqref{eq:radial:od} can be tightened as
		\[h_j \leq X^*_{kj},\quad \forall k \in \bar{\mcV}_\ell.\]
	\end{lemma}
	
	\begin{proof}
		The edge $(i,j)\in \hmcE$ is the only edge that connects the nodes in $\bar{\mcV}_\ell$ to the rest of the network. Hence, the common path to the reference bus of any two nodes in $\bar{\mcV}_\ell$ must include the path of node $j$ to the reference. It follows that the shortest path weight $h_j$ is a valid lower bound for all nodes in $\bar{\mcV}_\ell$.
	\end{proof}
	
	Identifying cutsets of cardinality larger than $1$ offers additional information to tighten the bounds of entries of the $X$ matrix. If graph $\hmcG$ gets disconnected upon removing lines $\ell_1,\ell_2\in\hmcE$, then at least one of these lines should be active. This logical conclusion translates to the constraint $z_{\ell_1}+z_{\ell_2}\geq 1$, which can be augmented to \eqref{eq:P4} to tighten the McCormick reformulation and possibly accelerate the MILP solver. Cutsets of larger cardinality, say $|\mathcal{C}| = k, k>1$, can be identified by iterating Steps 1 through 3 of the algorithm described earlier. In this case, we assign the weights of $k+\epsilon$ on the critical edges, where $\epsilon > 0$.  
	
	\section{Numerical Tests} \label{sec:nx}
	All tests were run on a 2.7 GHz Intel Core i5 laptop with 8GB RAM. The MILP formulations were solved using Gurobi v8.0.1 optimizer, written in Julia/JuMP ~\cite{gurobi, jump}.
	
	The performance of the MILP in \eqref{eq:P4} was tested for augmenting an existing network as well as for designing a radial one afresh. For the augmentation setup, the IEEE $39$-bus system benchmark was used as the pre-existing connected network~\cite{39bus}. The set $\hmcE$ was selected by $10$ randomly picked additional lines. From these lines, we solved the restricted version of \eqref{eq:P4} for $K =\{2,3, 4,5\}$. To satisfy Assumption~\ref{as:A1}, we assumed $M_i=10^{-4}$ on all buses that did not host generators, and $D_i=d=0.025$ for all $i \in \cal V$. To grade an arbitrarily constructed network, we compared its squared ${\cal H}_2$ norm to that of the optimal network of the same edge cardinality. The results are summarized in Table~\ref{table_1}.  Additional lines are useful in minimizing disturbances, and so the budget constraint in \eqref{eq:P4} was always met with equality. For all cases, the augmentation design problem was solved in less than 5 seconds. 
	
	\begin{table}[t]
		\centering
		\caption{Cost of designed grids using the IEEE $39$-bus network.}
		\begin{tabular}{lccc}
			\toprule
			\# of lines & Optimal Topology & Suboptimal Topology \\
			\cmidrule{1-3}
			$2$ &   $0.570$  & $0.585$  \\
			$3$  &  $0.564$  &$0.582$   \\
			$4$ &  $0.557$  &$0.576$     \\
			$5$ &  $0.552$ &$0.570$   \\
			\bottomrule
		\end{tabular}
		\label{table_1}
	\end{table}
	
	We next considered the radial topology design problem with $\hmcE$ composed of all edges in the IEEE $39$-bus network. The optimal cost obtained after solving the MILP in this case was $1.669$, and the time required to find the optimal tree was close to $2$ hours. Considering that the problem needs to be solved once off-line, this running time may not be of concern. Figure~\ref{fig:IEEE39} shows the optimal radial topology that was identified. 
	
	\begin{figure*}[t]
		\centering
		\includegraphics[scale=0.45]{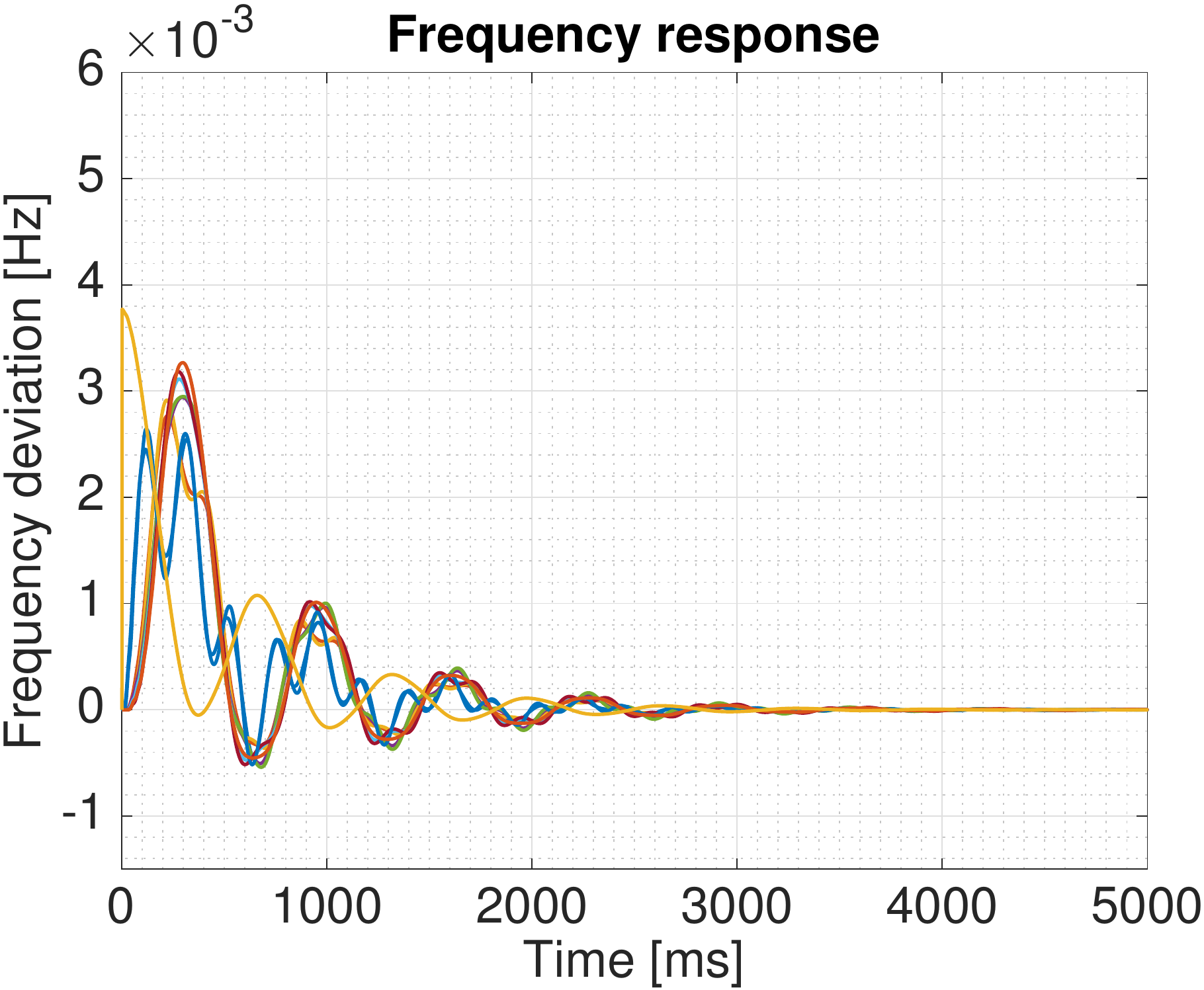}\hspace*{1em}
		\includegraphics[scale=0.45]{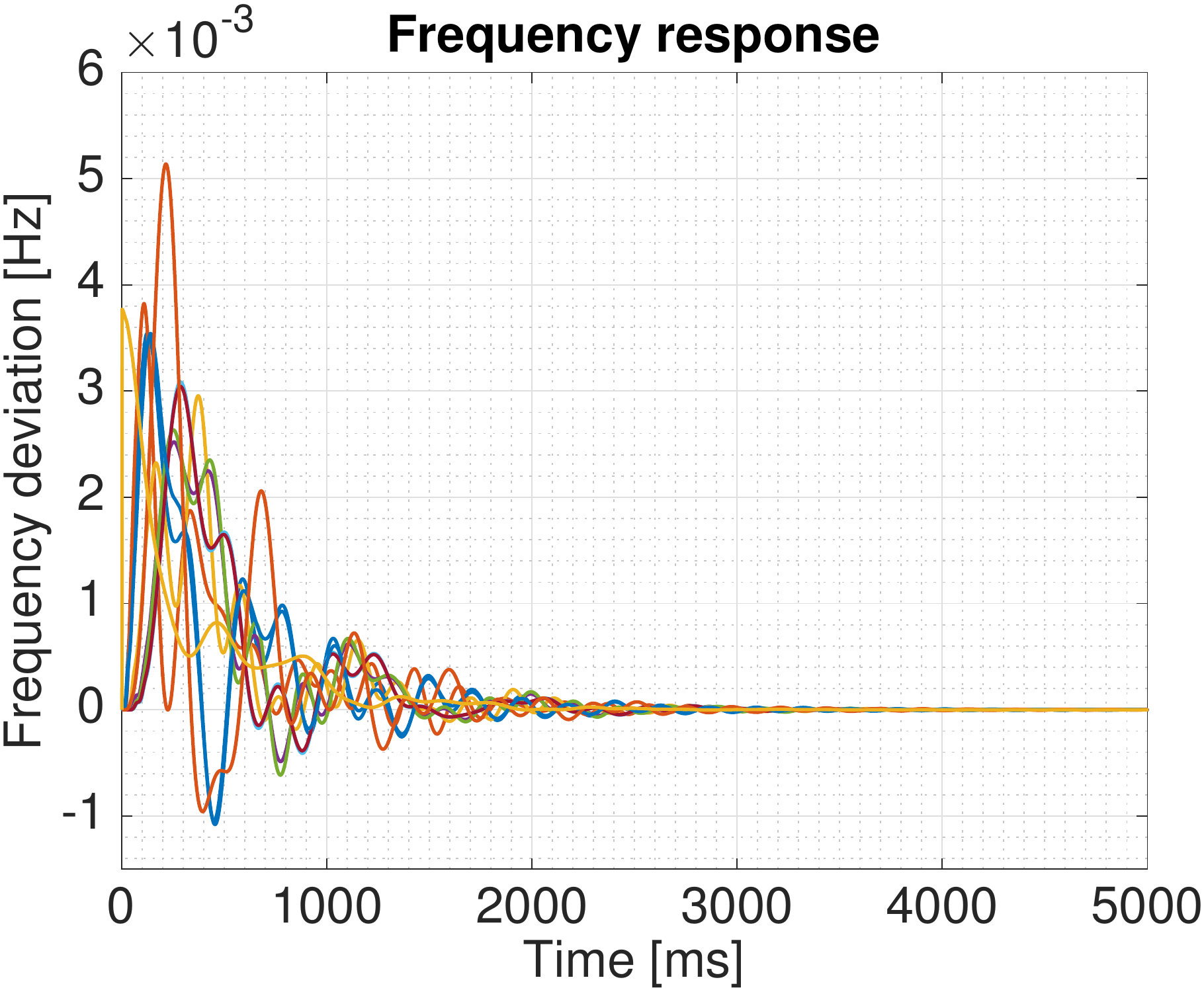}
		\caption{Frequency response at generator buses $30$--$39$ for a impulse input on bus 39: (a) Optimal tree (left panel); (b) Sub-optimal tree (right panel).}
		\label{fig:freqres}
	\end{figure*}
	
	\begin{figure}[t]
		\centering
		\includegraphics[scale=0.8]{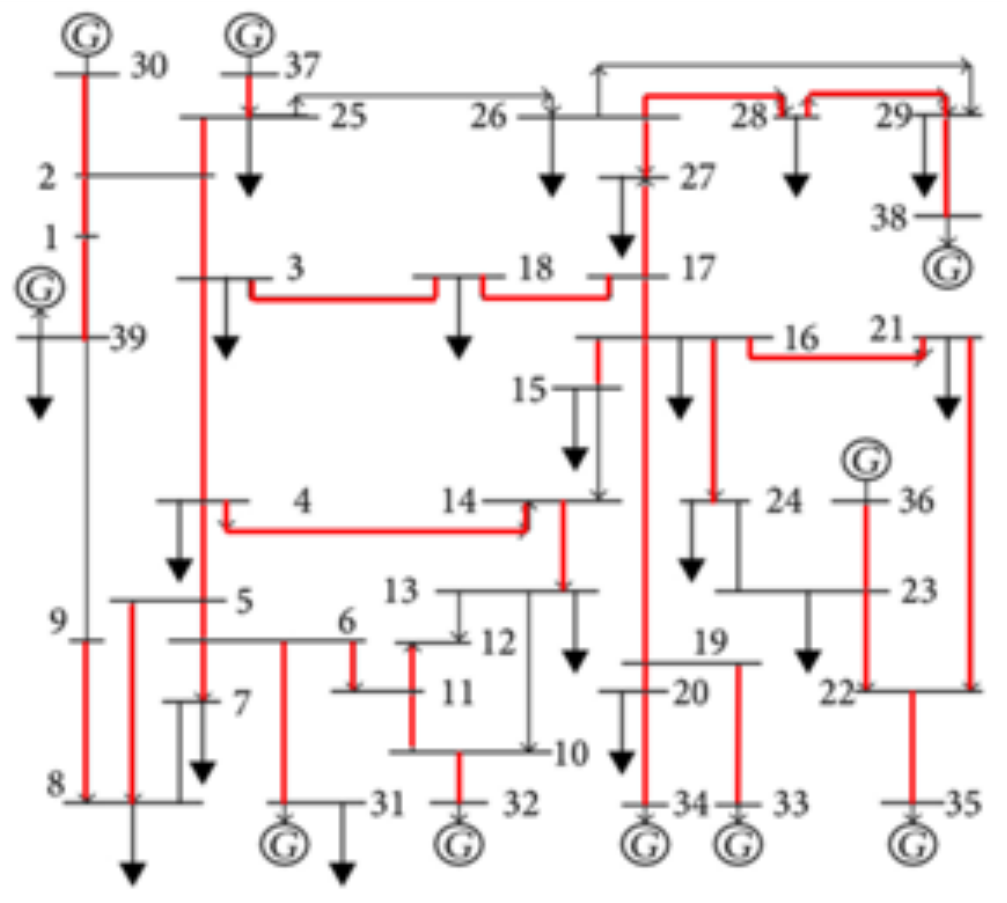}
		\caption{Optimal radial topology (bold red lines) considering all possible edges of the IEEE $39$-bus feeder \cite{39bus}.}
		\label{fig:IEEE39}
	\end{figure}
	
	Instead of using the bounds of Lemma~\ref{UBcor1} and the bound tightening procedure of Section~\ref{subsec:boundt}, we attempted to solve the same MILP with the relatively looser bounds of $0 \leq X^*_{ij} \leq 10$ for all $(i,j) \in \cal V$. In this case, the solver reached the optimality gap of 60\% after running for $3$ hours. Clearly, having tighter bounds improves the computation time. 
	
	For a 1 per unit impulse input at bus $39$, Figure~\ref{fig:freqres} compares the frequency response of the optimal tree (left panel) with an arbitrarily selected suboptimal tree (right panel). Notice that not only is the amplitude of oscillations lower in the left panel, but the generators also seem to drift together. These observations indicate that the optimal tree selected is much more effective at minimizing the effect of small disturbances.  
	
	\section{Conclusions and Future Work}
	\label{sec:conc}
	We have presented a general framework to study the effect of network topology on power grid dynamics. Using a system-theoretic approach, we have shown that the original topology design problem (MI-SDP) can be reformulated in tractable form as an MILP. To improve the computation time, graph-theoretic properties have been exploited to simplify the model and provide tighter bounds on the continuous optimization variables involved. Numerical tests on the IEEE 39-bus benchmark suggest that meshed networks exhibit improved coherence behavior. Current research efforts are focused on tackling the topology design task with non-uniform damping, exploring conditions for submodularity, and considering the design problem from an $\mathcal{H}_{\infty}$-norm perspective.
	
	\balance
	\bibliographystyle{IEEEtran}
	\bibliography{myabrv,references}
	
\end{document}